\documentclass[a4paper,12pt]{scrartcl}

\usepackage{amsmath,amssymb,amsthm}
\usepackage[T1]{fontenc}
\usepackage[utf8]{inputenc}
\usepackage{lmodern}
\usepackage{graphicx}
\usepackage{float}
\usepackage{subcaption}
\usepackage[dvipsnames,svgnames]{xcolor}
\usepackage{tikz,pgf}
\usetikzlibrary{arrows.meta}
\usetikzlibrary{calc}
\usepackage{booktabs}
\usepackage[colorlinks=true,linkcolor=RoyalBlue,citecolor=PineGreen,urlcolor=RoyalBlue]{hyperref}
\usepackage[nameinlink]{cleveref}
\usepackage{array}

\newcolumntype{L}{>{$}l<{$}}

\newcommand{\geng}{\textsc{Geng}}
\newcommand{\nauty}{\textsc{Nauty}}
\newcommand{\rigicomp}{\textsc{RigiComp}}
\newcommand{\mathematica}{\textsc{Mathematica}}

\newtheorem{theorem}{Theorem}

\newtheorem{lemma}[theorem]{Lemma}

\theoremstyle{definition}

\colorlet{colbg}{white}
\colorlet{colfg}{black}
\colorlet{colgraphv}{colfg!75!white}
\colorlet{colgraphe}{colfg!55!white}
\colorlet{colG}{DarkSeaGreen}
\definecolor{colR}{HTML}{CC6677}
\definecolor{colO}{HTML}{DDCC77}
\definecolor{colB}{HTML}{6699CC}
\colorlet{colY}{Gold!90!black}
\colorlet{colGray}{white!60!black}
\colorlet{colGold}{Gold!90!black}

\tikzstyle{vertex}=[fill=colgraphv,circle,inner sep=0pt, minimum size=4pt]
\tikzstyle{edge}=[line width=1.5pt,colgraphe]
\tikzstyle{edget}=[line width=1pt,colgraphe]
\tikzstyle{labelsty}=[font=\scriptsize]

\allowdisplaybreaks[1]
\pgfkeys{/pgf/number format/.cd,
int detect,
/pgf/number format/set thousands separator={\,}}

\begin{document}

\title{Minimal counterexamples to Hendrickson's conjecture on globally rigid graphs}

\author{%
  Georg Grasegger%
  \thanks{Johann Radon Institute for Computational and Applied
    Mathematics (RICAM), Austrian Academy of Sciences,
    Altenberger Straße 69, 4040 Linz, Austria}
}
\date{}

\maketitle

\begin{abstract}
    In this paper we consider the class of graphs which are redundantly $d$-rigid and $(d+1)$-connected but not globally $d$-rigid, where $d$ is the dimension.
    This class arises from counterexamples to a conjecture by Bruce Hendrickson.
    It seems that there are relatively few graphs in this class for a given number of vertices.
    Using computations we show that $K_{5,5}$ is indeed the smallest counterexample to the conjecture.
\end{abstract}

\begin{abstract}
    In this paper we consider the class of graphs which are redundantly $d$-rigid and $(d+1)$-connected but not globally $d$-rigid, where $d$ is the dimension.
    This class arises from counterexamples to a conjecture by Bruce Hendrickson.
    It seems that there are relatively few graphs in this class for a given number of vertices.
    Using computations we show that $K_{5,5}$ is indeed the smallest counterexample to the conjecture.
\end{abstract}

\section{Introduction}
The basis of the class of graph we consider is the following theorem from Hendrickson.
\begin{theorem}[Hendrickson \cite{Hendrickson}]
 A globally $d$-rigid graph with at least $d+2$ vertices is $(d+1)$-connected and redundantly $d$-rigid.
\end{theorem}
Hendrickson conjectured that the reverse direction of the theorem also holds.
While this is true for $d\in\{1,2\}$ \cite{JacksonJordan}, it was shown to be false in any higher dimension \cite{Connelly} using some bipartite graphs.
Since then further counterexamples were found \cite{FrankJiang,JordanKiralyTanigawa}.
Still very little is known on the set of graphs which are $(d+1)$-connected and redundantly $d$-rigid but not globally rigid.
In this paper we computationally show that the well-known $K_{5,5}$ is indeed the smallest counterexample in dimension three.

Before starting the proof, we recall some definitions and results from rigidity theory.
A framework $(G,\rho)$ consisting of a graph $G$ and a placement $\rho$ of the vertices in $d$-dimensional space is \emph{$d$-rigid} if edge-length preserving continuous motions can only arise from isometries in that space.
A graph is $d$-rigid, if the framework with any generic placement is $d$-rigid.
\emph{Redundant $d$-rigidity} is obtained if the graph is $d$-rigid and remains so after removal of any of its edges.
A graph is \emph{globally $d$-rigid}, if all frameworks which yield the same edge-lengths, indeed also yield the same distances between all pairs of non-adjacent vertices.
A graph that is redundantly $d$-rigid and $(d+1)$-connected but not globally $d$-rigid, i.\,e.\ a counterexample to Hendrickson's conjecture, we call an \emph{$H_d$-graph} in this paper.

If the number of vertices is comparably low with respect to the dimension then Hendrickson's conjecture has been proven to hold.
\begin{theorem}[{Jordán~\cite[Thm.~3.2]{JordanHigherDim}}]
 A graph $G=(V,E)$ with $d+2 \leq |V| \leq d+4$ is globally $d$-rigid if and only if it is $(d+1)$-connected and redundantly $d$-rigid.
\end{theorem}
In this paper we extend the theorem to $|V|\leq d+6$.

\section{Computations}
In this section we describe what computations we were able to do and what software we used for that.

Minimally rigid graphs with $n$ vertices in dimension $d$ do have $d n-\binom{d+1}{2}$ edges. We use \geng\ (a part of \nauty\ \cite{Nauty,NautyPaper}) to generate such graphs.
With a plugin by Martin Larsson \cite{Larsson} this process can be sped up using sparsity properties of rigid graphs, i.\,e.\ edge counts on subgraphs.
However, not all of these graphs are indeed $d$-rigid.
We then use \rigicomp\ \cite{RigiComp}, a \mathematica\ package, to check for rigidity.
Hence, we obtain lists of all minimally rigid graphs with a certain number of vertices for dimension $d\leq 23$.
We then can construct all $d$-rigid graphs by adding edges or we use \geng\ instead for getting graphs with at least $d n-\binom{d+1}{2}$ edges and check rigidity for all of them.
From those we can select the $(d+1)$-connected and redundantly $d$-rigid ones (see \Cref{tab:reddp1}).
Since redundantly rigid graphs have at least $d n-\binom{d+1}{2}+1$ edges, we can ignore all minimally rigid ones. The set of redundantly $d$-rigid graphs is made available at \cite{DataRed}.
\begin{table}[ht]
    \centering\scriptsize
    \begin{tabular}{Lr*{12}{r}}
        \toprule
        |V|   & $d=3$                            & $d=4$                            & $d=5$                            & $d=6$                            & $d=7$                            & $d=8$                            \\\midrule
        d+2   & \pgfmathprintnumber{1}           & \pgfmathprintnumber{1}           & \pgfmathprintnumber{1}           & \pgfmathprintnumber{1}           & \pgfmathprintnumber{1}           & \pgfmathprintnumber{1}           \\
        d+3   & \pgfmathprintnumber{3}           & \pgfmathprintnumber{3}           & \pgfmathprintnumber{3}           & \pgfmathprintnumber{3}           & \pgfmathprintnumber{3}           & \pgfmathprintnumber{3}           \\
        d+4   & \pgfmathprintnumber{19}          & \pgfmathprintnumber{22}          & \pgfmathprintnumber{23}          & \pgfmathprintnumber{24}          & \pgfmathprintnumber{24}          & \pgfmathprintnumber{24}          \\
        d+5   & \pgfmathprintnumber{304}         & \pgfmathprintnumber{531}         & \pgfmathprintnumber{756}         & \pgfmathprintnumber{900}         & \pgfmathprintnumber{984}         & \pgfmathprintnumber{1021}        \\
        d+6   & \pgfmathprintnumber{12055}       & \pgfmathprintnumber{47777}       & \pgfmathprintnumber{128855}      & \pgfmathprintnumber{246302}      & \pgfmathprintnumber{362922}      & \pgfmathprintnumber{448587}      \\
        d+7   & \pgfmathprintnumber{1079143}     & \pgfmathprintnumber{11623750}    & \pgfmathprintnumber{78033041}    &                                  &                                  &                                  \\
        \bottomrule
    \end{tabular}
    \begin{tabular}{Lr*{12}{r}}
        \toprule
        |V|   & $d=9$                          & $d=10$                         & $d=11$                         & $d=12$                         & $d=13$                         & $d=14$                         & $d=15$                         & $d=16$                         \\\midrule
        d+2   & \pgfmathprintnumber{1}         & \pgfmathprintnumber{1}         & \pgfmathprintnumber{1}         & \pgfmathprintnumber{1}         & \pgfmathprintnumber{1}         & \pgfmathprintnumber{1}         & \pgfmathprintnumber{1}         & \pgfmathprintnumber{1}         \\
        d+3   & \pgfmathprintnumber{3}         & \pgfmathprintnumber{3}         & \pgfmathprintnumber{3}         & \pgfmathprintnumber{3}         & \pgfmathprintnumber{3}         & \pgfmathprintnumber{3}         & \pgfmathprintnumber{3}         & \pgfmathprintnumber{3}         \\
        d+4   & \pgfmathprintnumber{24}        & \pgfmathprintnumber{24}        & \pgfmathprintnumber{24}        & \pgfmathprintnumber{24}        & \pgfmathprintnumber{24}        & \pgfmathprintnumber{24}        & \pgfmathprintnumber{24}        & \pgfmathprintnumber{24}        \\
        d+5   & \pgfmathprintnumber{1040}      & \pgfmathprintnumber{1047}      & \pgfmathprintnumber{1051}      & \pgfmathprintnumber{1052}      & \pgfmathprintnumber{1053}      & \pgfmathprintnumber{1053}      & \pgfmathprintnumber{1053}      & \pgfmathprintnumber{1053}      \\
        d+6   & \pgfmathprintnumber{498232}    & \pgfmathprintnumber{522517}    & \pgfmathprintnumber{533092}    & \pgfmathprintnumber{537457}    & \pgfmathprintnumber{539203}    & \pgfmathprintnumber{539912}    & \pgfmathprintnumber{540194}    & \pgfmathprintnumber{540313}    \\
        \bottomrule
    \end{tabular}
    \begin{tabular}{Lr*{12}{r}}
        \toprule
        |V|   & $d=17$                         & $d=18$                         & $d=19$                         & $d=20$                         & $d=21$                         & $d=22$                         & $d=23$                         \\\midrule
        d+2   & \pgfmathprintnumber{1}         & \pgfmathprintnumber{1}         & \pgfmathprintnumber{1}         & \pgfmathprintnumber{1}         & \pgfmathprintnumber{1}         & \pgfmathprintnumber{1}         & \pgfmathprintnumber{1}         \\
        d+3   & \pgfmathprintnumber{3}         & \pgfmathprintnumber{3}         & \pgfmathprintnumber{3}         & \pgfmathprintnumber{3}         & \pgfmathprintnumber{3}         & \pgfmathprintnumber{3}         & \pgfmathprintnumber{3}         \\
        d+4   & \pgfmathprintnumber{24}        & \pgfmathprintnumber{24}        & \pgfmathprintnumber{24}        & \pgfmathprintnumber{24}        & \pgfmathprintnumber{24}        & \pgfmathprintnumber{24}        & \pgfmathprintnumber{24}        \\
        d+5   & \pgfmathprintnumber{1053}      & \pgfmathprintnumber{1053}      & \pgfmathprintnumber{1053}      & \pgfmathprintnumber{1053}      & \pgfmathprintnumber{1053}      & \pgfmathprintnumber{1053}      & \pgfmathprintnumber{1053}      \\
        d+6   & \pgfmathprintnumber{540360}    & \pgfmathprintnumber{540381}    & \pgfmathprintnumber{540389}    & \pgfmathprintnumber{540393}    & \pgfmathprintnumber{540394}    & \pgfmathprintnumber{540395}    & \pgfmathprintnumber{540395}    \\
        \bottomrule
    \end{tabular}
    \caption{Number of redundantly $d$-rigid and $(d+1)$-connected graphs for different dimensions.}
    \label{tab:reddp1}
\end{table}

We then compute, again with \rigicomp, the globally $d$-rigid ones from these lists (see \Cref{tab:glob}). The list of globally rigid graphs is available at \cite{DataGlob}.
\begin{table}[ht]
    \centering\scriptsize
    \begin{tabular}{Lr*{12}{r}}
        \toprule
        |V|   & $d=3$                            & $d=4$                            & $d=5$                            & $d=6$                            & $d=7$                            & $d=8$                            \\\midrule
        d+2   & \pgfmathprintnumber{1}           & \pgfmathprintnumber{1}           & \pgfmathprintnumber{1}           & \pgfmathprintnumber{1}           & \pgfmathprintnumber{1}           & \pgfmathprintnumber{1}           \\
        d+3   & \pgfmathprintnumber{3}           & \pgfmathprintnumber{3}           & \pgfmathprintnumber{3}           & \pgfmathprintnumber{3}           & \pgfmathprintnumber{3}           & \pgfmathprintnumber{3}           \\
        d+4   & \pgfmathprintnumber{19}          & \pgfmathprintnumber{22}          & \pgfmathprintnumber{23}          & \pgfmathprintnumber{24}          & \pgfmathprintnumber{24}          & \pgfmathprintnumber{24}          \\
        d+5   & \pgfmathprintnumber{304}         & \pgfmathprintnumber{531}         & \pgfmathprintnumber{756}         & \pgfmathprintnumber{900}         & \pgfmathprintnumber{984}         & \pgfmathprintnumber{1021}        \\
        d+6   & \pgfmathprintnumber{12055}       & \pgfmathprintnumber{47777}       & \pgfmathprintnumber{128855}      & \pgfmathprintnumber{246302}      & \pgfmathprintnumber{362922}      & \pgfmathprintnumber{448587}      \\
        d+7   & \pgfmathprintnumber{1079142}     & \pgfmathprintnumber{11623749}    & \pgfmathprintnumber{78033040}    &                                  &                                  &                                  \\
        \bottomrule
    \end{tabular}
    \caption{Number of globally $d$-rigid graphs for different dimensions. For $|V|\leq d+6$ this is the same as \Cref{tab:reddp1}.}
    \label{tab:glob}
\end{table}
Since checking redundant rigidity seems to be the slowest computation we may also first compute rigid, then the non-globally rigid ones.
From those we check connectivity and only in the end we compute redundancy.

Note, that \rigicomp\ provides both a probabilistic relatively fast computation which might give false negatives and a symbolic implementation (both based on \cite{GortlerHealyThurston}).
For $|V|\leq d+5$ we can use the symbolic approach. For more vertices we used the probabilistic method.
Nevertheless, for $|V|\leq d+6$ we get an exact answer due to the following.

Let us first assume we want to have all minimally rigid graphs.
Every minimally rigid graph is tight which means it has $d|V|-\binom{d+1}{2}$ edges and every induced subgraph with $|V'|\geq d$ vertices has at most $d|V'|-\binom{d+1}{2}$ edges.
These tight graphs are those that we get from initial computations from \geng\ with the plugin \cite{Larsson}.
Not all of them, however, are rigid. The remaining ones are either flexible circuits or 0-extensions thereof or contain a non-tight circuit as a spanning subgraph.
In \cite{FlexibleCircuits} flexible circuits with at most $d+6$ vertices have been fully classified.
Hence, we have a deterministic answer for minimal rigidity in that case.
Computationally we may first use the probabilistic procedure and in case of a negative answer do check via flexible circuits.
We can do similarly when checking $d$-rigidity (i.\,e.\ containing a spanning minimally rigid graph) and redundant rigidity by tracing back to minimal rigidity.
Hence, for $|V|\leq d+6$ the results are deterministic.
For $d=3$ also the computation for $|V|=d+7$ was done symbolically.

The computations show that all redundantly $d$-rigid graphs that are also $(d+1)$-connected are indeed globally $d$-rigid if $|V|\leq d+6$.
For $|V|=d+7$ we see that $K_{5,5}$ and successive cones give the only $H_d$-graphs at least until $d=5$.
The \emph{cone} of a graph is obtained by adding a new vertex and edges from it to all the existing vertices (compare \cite{ConingR,ConingG}).
This construction we also call \emph{coning}.

\section{Main Results}
From the computations and some edge counting we can state the following main theorem.
\begin{theorem}\label{thm:main}
 A graph $G=(V,E)$ with $d+2 \leq |V| \leq d+6$ is globally $d$-rigid if and only if it is $(d+1)$-connected and redundantly $d$-rigid.
\end{theorem}

For the proof we use two basic statements.
\begin{lemma}\label{lem:stat}
 Let $G=(V,E)$ be a redundantly $d$-rigid graph with $d\geq23$, $n=|V|=d+k$ and $k\leq 6$.
 Then $G$ has a vertex of degree $n-1$.
\end{lemma}
\begin{proof}
 Since $G$ is redundantly $d$-rigid it has $dn-\binom{d+1}{2}+\ell$ edges for some $\ell\geq 1$.
 We assume for contradiction that all vertices have degree less than $n-1$.
 Then
 \begin{align*}
  0&= \sum_{v\in V} \deg(v) -2|E|\\
   &\leq n(n-2)- 2(dn - \binom{d+1}{2}+\ell)\\
   &= (d+k)(d+k-2) -2d(d+k) +d(d+1) -2\ell\\
   &= (d+k)(k-d-2) +d(d+1)-2\ell = k^2-d^2-2d-2k +d^2 +d -2\ell \\
   &= k^2-2k-d-2\ell
   \leq 24-d-2\ell
 \end{align*}
 But this cannot be true for $d\geq 23$ and therefore $G$ has a vertex of degree $n-1$.
\end{proof}

Let $G$ be a graph and let $G'$ be constructed from $G$ by coning.
Coning preserves several properties in rigidity while increasing the dimension.
Rigidity is proven to be preserved by \cite{ConingR}. This means that a graph is $d$-rigid if and only if the cone is $(d+1)$-rigid.
The same is true for global rigidity due to \cite{ConingG}.
If the cone $G'$ is redundantly $(d+1)$-rigid, then it is easy to see that $G$ is redundantly $d$-rigid.
From this we get the following Lemma.

\begin{lemma}\label{lem:uncone}
 Let $d\geq 23$ and let $G=(V,E)$ be a redundantly $d$-rigid and $(d+1)$-connected graph with $n=|V|=d+k$ and $k\leq 6$.
 If $G$ is not globally $d$-rigid then there exists an $H_{d-1}$-graph with $n-1$ vertices.
\end{lemma}
\begin{proof}
 By \Cref{lem:stat} we know that $G$ has a vertex of degree $n-1$.
 Hence, removing this vertex gives a graph $G'$ that is redundantly $(d-1)$-rigid by and $d$-connected but not globally $(d-1)$-rigid on $n-1$ vertices.
\end{proof}

Together with the computational results we can finally proof the main theorem.
\begin{proof}[Proof of \Cref{thm:main}]
 Let $G$ be a redundantly $d$-rigid and $(d+1)$-connected graph with $|V|\leq d+6$.
 If $d\leq 22$, then we have seen from the computations that $G$ is globally $d$-rigid.
 Assuming $d\geq 23$ we know from \Cref{lem:uncone} that if $G$ was not globally $d$-rigid then the removal of a vertex with degree $n-1$, which does exist, gives a non-globally $(d-1)$-rigid graph that is $d$-connected and redundantly $(d-1)$-rigid graph.
 If $d-1=22$ we know from the computations that such a graph does not exist.
 If $d-1\geq23$ we continue inductively.
 This proves that $G$ is globally $d$-rigid.
\end{proof}

\section{Examples}
It is known from \cite{Connelly} that $K_{5,5}$ is an $H_3$-graph. By successive coning we get $H_d$-graphs for all $d\geq3$ by \cite{ConingG}.
More generally, it is known from \cite{Connelly} that $K_{m,n}$ is an $H_d$-graph if $m+n=\binom{d+2}{2}$ and $m,n\geq d+2$.
This means that there are two $H_4$-graphs $K_{6,9}$ and $K_{7,8}$ with $15$ vertices each, which is out of the computational scope of this paper.

For $d=3$ a construction of six globally $3$-rigid graphs, glued in a cycle (see for example \Cref{fig:H3:K5cyc}) gives an $H_3$-graph due to \cite{JordanKiralyTanigawa}.
In \Cref{fig:H3:K55cyc} we show an example of such a gluing with $K_{5,5}$, i.\,e., a non-globally $3$-rigid base graph, which also gives an $H_3$-graph as shown by computations.
Both these graphs are not minimal in terms of edge counts with respect to the $H_3$ property.
Indeed we can find five edges (indicated in red in the figures) that we can remove such that the result is still an $H_3$-graph.
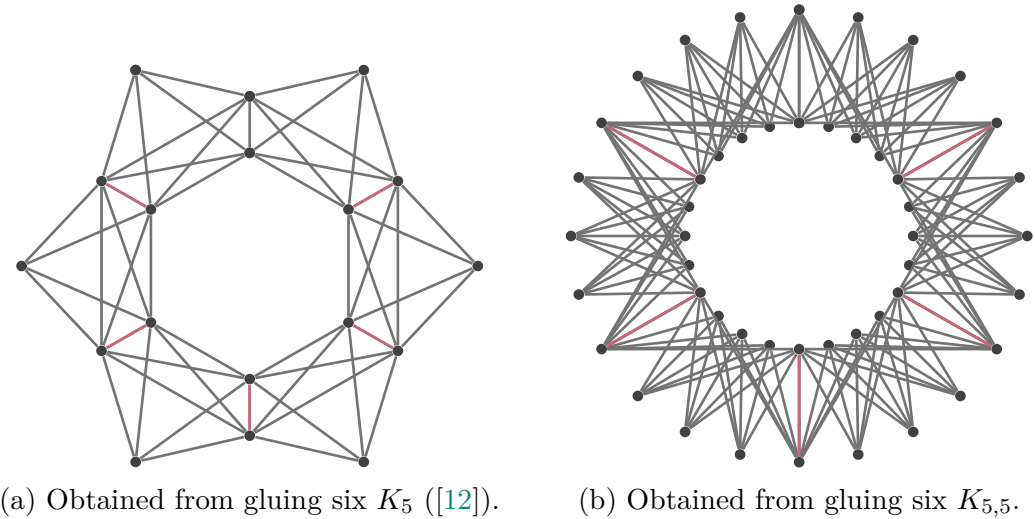
\begin{figure}[ht]
    \centering
    \begin{subfigure}[b]{0.45\textwidth}
        \centering
        \begin{tikzpicture}[scale=0.75,rotate=90]
            \foreach \w [count=\i from 0] in {0,60,...,360}
            {
                \node[vertex] (a\i) at (\w:2) {};
                \node[vertex] (b\i) at (\w:3) {};
                \node[vertex] (c\i) at (\w+30:4) {};
            }
            \foreach \i [evaluate=\i as \ip using {int(mod(int(\i+1),6))}] in {0,1,2,3,4,5}
            {
                \draw[edget] (a\i)edge(b\i);
                \draw[edget] (a\i)edge(c\i);
                \draw[edget] (a\i)edge(b\ip);
                \draw[edget] (a\i)edge(a\ip);
                \draw[edget] (a\ip)edge(b\i);
                \draw[edget] (a\ip)edge(c\i);
                \draw[edget] (b\i)edge(b\ip);
                \draw[edget] (b\i)edge(c\i);
                \draw[edget] (b\ip)edge(c\i);
            }
            \draw[edget,colR] (a1)edge(b1) (a2)edge(b2) (a3)edge(b3) (a4)edge(b4) (a5)edge(b5);
        \end{tikzpicture}
        \caption{Obtained from gluing six $K_5$ (\cite{JordanKiralyTanigawa}).}
        \label{fig:H3:K5cyc}
    \end{subfigure}
    \begin{subfigure}[b]{0.45\textwidth}
        \centering
        \begin{tikzpicture}[scale=0.75,rotate=90]
            \foreach \w [count=\i] in {0,15,...,360}
            {
                \node[vertex] (a\i) at (\w:2) {};
                \node[vertex] (b\i) at (\w:4) {};
            }
            \foreach \d in {0,1,2,3,4}
            {
                \foreach \k [evaluate=\k as \kk using int(\k+\d)] in {1,5,...,21}
                {
                    \foreach \i [count=\ii from \kk-\d] in {0,1,2,3,4}
                    {
                        \draw[edget] (a\kk)edge(b\ii);
                    }
                }
            }
            \draw[edget,colR] (a5)edge(b5) (a9)edge(b9) (a13)edge(b13) (a17)edge(b17) (a21)edge(b21);
        \end{tikzpicture}
        \caption{Obtained from gluing six $K_{5,5}$.}
        \label{fig:H3:K55cyc}
    \end{subfigure}
    \caption{Two $H_3$-graphs.}
    \label{fig:H3}
\end{figure}

\section{Conclusion}
While the computations were handy to show that there are no other small counterexamples to Hendrickson's conjecture other than the known ones,
they are insufficient for a full classification. Nevertheless, we think that computational experiments might reveal new counterexamples.
For this, however either some more theory or improved software/hardware is needed.

\section{Acknowledgments}
This work was partially supported by the Austrian Science Fund (FWF): P31888.

\end{document}